\documentclass[reqno]{amsart}
\usepackage{enumerate}
\usepackage{scrextend}
\changefontsizes[16pt]{12pt}

\newtheorem{theorem}{Theorem}[section]
\newtheorem{lemma}[theorem]{Lemma}

\theoremstyle{definition}
\newtheorem{definition}[theorem]{Definition}

\newtheorem{remark}[theorem]{Remark}

\newtheorem{assumption}[theorem]{Assumption}

\newcommand{\norm}[1]{\left\Vert#1\right\Vert}

\numberwithin{equation}{section}
\usepackage{amsfonts}

\usepackage{amsmath}
\usepackage{amssymb}

\begin{document}
\font\nho=cmr10
\def\dive{\mathrm{div}}
\def\cal{\mathcal}
\def\L{\cal L}

\def \ud{\underline }
\def\id{{\indent }}
\def\f{\frac}
\def\non{{\noindent}}
 \def\le{\leqslant} 
 \def\leq{\leqslant}
 \def\geq{\geqslant} 
\def\rar{\rightarrow}
\def\Rar{\Rightarrow}
\def\ti{\times}
\def\i{\mathbb I}
\def\j{\mathbb J}
\def\si{\sigma}
\def\Ga{\Gamma}
\def\ga{\gamma}
\def\ld{{\lambda}}
\def\Si{\Psi}
\def\f{\mathbf F}
\def\r{\hro{R}}
\def\e{\cal{E}}
\def\B{\cal B}
\def\A{\mathcal{A}}
\def\p{\mathbb P}

\def\tet{\theta}
\def\Tet{\Theta}
\def\hro{\mathbb}
\def\ho{\mathcal}
\def\P{\ho P}
\def\E{\mathcal{E}}
\def\n{\mathbb{N}}
\def\M{\mathbb{M}}
\def\dMu{\mathbf{U}}
\def\dMcs{\mathbf{C}}
\def\dMcu{\mathbf{C^u}}
\def\vk{\vskip 0.2cm}
\def\td{\Leftrightarrow}
\def\df{\frac}
\def\Wei{\mathrm{We}}
\def\Rey{\mathrm{Re}}
\def\s{\mathbb S}
\def\l{\mathcal{L}}
\def\C+{C_+([t_0,\infty))}
\def\o{\cal O}

\title[AAP-solutions of Parabolic equations on hyperbolic manifolds]{Asymptotically almost periodic solutions to parabolic equations on the real hyperbolic manifold}


\author[P.T. Xuan]{Pham Truong Xuan}
\address{Pham Truong Xuan \hfill\break
Corresponding author, Faculty of Computer Science and Engineering,  Department of Mathematics, Thuyloi university \hfill\break
Khoa Cong nghe Thong tin, Bo mon Toan, Dai hoc Thuy loi, 175 Tay Son, Dong Da, Ha Noi, Viet Nam} 
\email{xuanpt@tlu.edu.vn or phamtruongxuan.k5@gmail.com}

\author[N.T. Van]{Nguyen Thi Van}
\address{Nguyen Thi Van \hfill\break
Faculty of Computer Science and Engineering,  Department of Mathematics, Thuyloi university\hfill\break
Khoa Cong nghe Thong tin, Bo mon Toan, Dai hoc Thuy loi, 175 Tay Son, Dong Da, Ha Noi, Viet Nam}
\email{van@tlu.edu.vn}


\author[B. Quoc]{Bui Quoc}
\address{Bui Quoc \hfill\break
Department of Mathematics-Mechanics-Informatics, Hanoi University of Science \hfill\break
Khoa Toan-Co-Tin, Dai hoc khoa hoc tu nhien, Dai hoc Quoc gia Ha Noi, 336 Nguyen Trai, Thanh Xuan, Ha Noi, Viet Nam}
\email{buiquoc\_t62@hus.edu.vn}

\begin{abstract}  
In this work we study the existence and the asymptotic behaviour of the asymptotically almost periodic mild solutions of the vectorial parabolic equations on the real hyperbolic manifold $\mathbb{H}^d(\mathbb{R})$ ($d \geqslant 2$). We will consider the vectorial laplace operator in the sense of Ebin-Marsden's laplace operator. Our method is based on certain dispertive and smoothing estimates of the semigroup generated by the linearized vectorial heat equation and the fixed point argument. First, we prove the existence and the uniqueness of the asymptotically almost periodic mild solution for the linearized equations. Next, using the fixed point argument, we can pass from linearized equations to semilinear equations to prove the existence, uniqueness, exponential decay and stability of the solutions. Our abstract results will be applied to the incompressible Navier-Stokes equation and the semilinear vectorial heat equation.
\end{abstract}



\maketitle

\tableofcontents

\section{Introduction}
In 1970, Ebin and Marsden introduced in \cite{EbiMa} the formula of the Navier-Stokes equation on an Einstein manifold with the negative Ricci curvature, where a real hyperbolic manifold $(\mathbb{H}^d(\mathbb{R}),g) \, (d\geq 2)$ is one specific case. They given in \cite{EbiMa} the generalized vectorial Laplace operator $L$ acting on vector fields $u$ by the mean of the deformation tensor
$$Lu := \frac{1}{2}\mathrm{div}(\nabla u + \nabla u^t)^{\sharp}.$$
The relations between the Ebin-Marsden's laplace operator $L$ and the Bochner-laplacian $\overrightarrow{\Delta}$ is given by
$$Lu = \overrightarrow{\Delta}u + R(u)$$
where $R$ is the Ricci operator defined on manifolds (in detail see Section \ref{2}).

Since then, this notion has been used in the works of Czubak
and Chan \cite{Cz1,Cz2,Cz3} and also Lichtenfelz \cite{Li2016} to prove the non-uniqueness of weak Leray solution of Navier-
Stokes equation on the three-dimensional hyperbolic manifolds. Furthermore, Pierfelice \cite{Pi} has proved the
dispersive and smoothing estimates ($L^p-L^q$-smoothing properties) for Stokes semigroups on the generalized non-compact manifolds with
negative Ricci curvature then combines these estimates with Kato-iteration method to prove the existence
and uniqueness of strong mild solutions to Navier-Stokes equations. For some related works on the compact manifolds with the equations associated with the Ebin-Marsden's laplace operator, we refer the reader to \cite{Fa2018,Fa2020,Li2016,MiTa2001,Sa}.

The existence, uniqueness and stability of the periodic mild solution for the incompressible Navier-Stokes equation on the non-compact Einstein manifolds with negative Ricci curvature have obtained in the recent work \cite{HuyXuan2018}. The method is based on the $L^p-L^q$-smoothing properties of the Stokes semigroup and the Massera-type principle.

In the present paper, inspiring from \cite{HuyXuan2018} we extend to study the existence and asymptotic behaviour of the asymptotically almost periodic solutions of  the generalized vectorial parabolic equations on the real hyperbolic manifold ${\bf M} :=\mathbb{H}^d(\mathbb{R})\, (d\geq 2)$. In particular, we study the following abstract parabolic equations
\begin{align}\label{CauchyHeat0}
\begin{cases}
\partial_t u = Lu + BG(u,t),\cr
u|_{t=0}=u_0 \in Y(\Gamma(T{\bf M})),
\end{cases}
\end{align}
where $L$ is Ebin-Marsden's laplace operator, $B:X\to Y$ is connection operator and $G(u,t):Y\times \mathbb{R}\to X$ is the nonlinear part satisfying some conditions (see Section 2). If $B=\mathbb{P}\dive,\,G(u,t) = (u\otimes u +F)(t)$, then Equation \eqref{CauchyHeat0} reduces to the Navier-Stokes equation (see Section \ref{4.1}). If $B= Id$ (identity operator), $G(u,t)=|u|^{k-1}u + f(t), \, (k\geqslant 2)$, then Equation \eqref{CauchyHeat0} reduces to the semilinear vectorial heat equation (see Section \ref{4.2}). 

On the other hand, we would like to note that the semilinear scalar heat equation associated with the scalar Laplace-Beltrami operator was studied on the hyperbolic manifolds in some recent works \cite{BaTe,MaTe,No,Va2018} and references therein.

The linear equation correspoding to \eqref{CauchyHeat0} is
\begin{align}\label{L}
\begin{cases}
\partial_t u = Lu + Bf(t),\cr
u|_{t=0}=u_0 \in Y(\Gamma(T{\bf M})).
\end{cases}
\end{align}
Since the sectional curvature on $({\bf M},g)$ is constant and negative (more precisely, it equals to $-(d-1)$), Pierfelice \cite{Pi} proved the $L^p-L^q$-smoothing properties for the vectorial heat semigroup, associated with the Ebin-Marsden's operator. Using these estimates, we prove the Massera-type principle that if $f$ is asymptotically almost periodic, then there exists an unique asymptotically almost periodic mild solution to the linear equation \eqref{L} (Section \ref{3.1}). Next, by the fixed point arguments we can pass from the linear equation to the semilinear equation \eqref{CauchyHeat0} to prove the existence of the asymptotically almost perioidc mild solutions of such equations. We also establish the exponential decay and stability of such solutions by using the cone inequality (Section \ref{3.2}). Finally, our abstract results will be applied to the Navier-Stokes equation and the semilinear vectorial heat equation (Section \ref{4}).

\medskip
{\bf Notations.} Throughout this paper we use the following notations:\\ 
$\bullet$ Let $X$ be a Banach space, we denote the Banach space of the continuous functions from $\mathbb{R}_+$ to $X$ by 
$$C_b(\r_+, X):=\{f:\r_+ \to X \mid f\hbox{ is continuous on $\mathbb{R}_+$ and }\sup_{t\in\r_+}\|f(t)\|_X<\infty\}$$
with the norm $\|f\|_{C_b(\r_+, X)}:=\sup_{t\in\r_+}\|f(t)\|_X.$\\
We denote also the Banach space of the continuous functions from $\mathbb{R}$ to $X$ by  
$$C_b(\r, X):=\{f:\r \to X \mid f\hbox{ is continuous on $\r$ and }\sup_{t\in\r}\|f(t)\|_X<\infty\}$$
with the norm $\|f\|_{C_b(\r, X)}:=\sup_{t\in\r}\|f(t)\|_X$.\\
$\bullet$ We denote the Levi-Civita connection by $\nabla$ and the set of vector field by $\Gamma(T{\bf M})$ and the set of tensor fields with second order by $\Gamma(T{\bf M}\otimes T{\bf M})$.\\ 
$\bullet$ We use the norms of tensor fields $T$ on the manifold $({\bf M},g)$:
$$\norm{T}_{L^p(\Gamma(T{\bf M}))} = \left( \int_{\bf M} |T|^p \mathrm{dVol}_{\bf M} \right)^{1/p},\hbox{   } |T|=\left< T,T \right>_g^{1/2}.$$

\section{Vectorial parabolic equations on the real hyperbolic manifold}\label{2}
Let  $({\bf M}=:\mathbb{H}^d(\mathbb{R}),g)$  be a real hyperbolic manifold of dimension $d\geq 2$ which is realized as the upper sheet 
$$x_0^2-x_1^2-x_2^2...-x_d^2 = 1 \,  \,( x_0\geq 1),$$
of hyperboloid in $\mathbb{R}^{d+1}$, equipped with the Riemannian metric 
$$g := -dx_0^2 + dx_1^2 + ... + dx_d^2.$$
In geodesic polar coordinates, the hyperbolic manifold is 
$$\mathbb{H}^d(\mathbb{R}): = \left\{ (\cosh \tau, \omega \sinh \tau), \, \tau\geq 0, \omega \in \mathbb{S}^{d-1}  \right\}$$
with the metric 
$$g =: d\tau^2+(\sinh\tau)^2d\omega^2$$
where  $d\omega^2$ is the canonical metric on the sphere $\mathbb{S}^{d-1}$.  
A remarkable property on $M$ is the Ricci curvature tensor : $\mathrm{Ric}_{ij}=-(d-1)g_{ij}$. 

Ebin and Marsden introduced the vectorial Laplacian $L$ on vector field $u$ by using the deformation tensor (see \cite{EbiMa} and more detail in \cite{Tay,Pi}):
$$Lu := \frac{1}{2}\mathrm{div}(\nabla u + \nabla u^t)^{\sharp},$$
where $\omega^{\sharp}$ is a vector field associated with the 1-form $\omega$ by $g(\omega^{\sharp},Y) = \omega(Y) \, \forall Y \in \Gamma(T{\bf M})$.
Since $\mathrm{div}\, u=0$ , $L$ can be expressed as 
$$Lu = \overrightarrow{\Delta}u + R(u),$$
where $\overrightarrow{\Delta}u =- \nabla^*\nabla u= \mathrm{Tr}_g(\nabla^2u)$ is the Bochner-Laplacian
and $R(u)=(\mathrm{Ric}(u,\cdot))^{\sharp}$ is the Ricci operator. Since $\mathrm{Ric}(u,\cdot)=-(d-1)g(u,\cdot)$, we have $R(u)=-(d-1)u$ and
$$Lu = \overrightarrow{\Delta}u -(d-1)u.$$

Putting $\mathcal{A}u: = -Lu$, we have that $e^{-t\A}$ is the semigroup associated with the homogeneous Cauchy problem of the vectorial heat equation
$$\partial_t u = -\mathcal{A}u.$$
We now recall the $L^p-L^q$ dispersive and smoothing properties of the semigroup $e^{-t\A}$.
\begin{lemma}{\rm(\cite[Theorem 4.1, Corollary 4.3]{Pi}):}\label{estimates}
\item[i)] For $t>0$, and $p$, $q$ such that $1\leq p \leq q \leq \infty$, 
the following dispersive estimates hold: 
\begin{equation}\label{dispersive}
\left\| e^{-t\mathcal{A}} u_0\right\|_{L^q} \leq [h_d(t)]^{\frac{1}{p}-\frac{1}{q}}e^{-t(d-1 + \gamma_{p,q})}\left\| u_0 \right\|_{L^p}\hbox{ for all }u_0 \in L^p(\Gamma(T{\bf M}))
\end{equation}
 where $h_d(t): = C\max\left( \frac{1}{t^{d/2}},1 \right)$, 
   $\gamma_{p,q}:=\frac{\delta_d}{2}\left[ \left(\frac{1}{p} - \frac{1}{q} \right) + \frac{8}{q}\left( 1 - \frac{1}{p} \right) \right]$ and $\delta_d$ are positive constants depending only on $d$.  
\item[ii)] For $p$ and $q$ such that $1<p\leq q <\infty$ we obtain for all $t>0$:
\begin{equation}
\left\| \nabla e^{-t\mathcal{A}}u_0 \right\|_{L^q} \leq [h_d(t)]^{\frac{1}{p}-\frac{1}{q}+\frac{1}{d}}e^{-t\left(d-1  + \frac{\gamma_{q,q}+\gamma_{p,q}}{2} \right)} \left\| u_0 \right\|_{L^p}
\end{equation}
\begin{equation}
\left\| e^{-t\mathcal{A}}\nabla^* T_0 \right\|_{L^q} \leq [h_d(t)]^{\frac{1}{p}-\frac{1}{q}+\frac{1}{d}}e^{-t\left(d-1  + \frac{\gamma_{q,q}+\gamma_{p,q}}{2} \right)} \left\| T_0 \right\|_{L^p}
\end{equation}
for all $u_0\in L^p(\Gamma(T{\bf M}))$ and all 
 tensor $T_0 \in L^p(\Gamma(T{\bf M} \otimes T^*{\bf M}))$. 
\item[iii)]  As a consequence of (ii) we obtain  for $t>0$ that
\begin{equation}
\left\| e^{-t\mathcal{A}}\mathrm{div}T^{\sharp}_0 \right\|_{L^q} \leq [h_d(t)]^{\frac{1}{p}-\frac{1}{q}+\frac{1}{d}}e^{-t\left(d-1 + \frac{\gamma_{q,q}+\gamma_{p,q}}{2} \right)} \left\| T_0^{\sharp}\right\|_{L^p}
\end{equation}
for all tensor $T^{\sharp}_0 \in L^p(\Gamma(T{\bf M}\otimes T{\bf M}))$.
\end{lemma}

For convenient to state and prove the abstract results, we rewrite the inequalities in Assertions $i)$ and $iii)$ in the generalized forms:
Putting $B=\dive$ or $\mathrm{Id}$, we have the following inequalities
\begin{equation}\label{estimates1}
\norm{e^{-t\cal{A}}u_0}_{Y(\Gamma(T{\bf M}))} \leqslant e^{-\sigma t} \norm{u_0}_{Y(\Gamma(T{\bf M}))},
\end{equation}
\begin{equation}\label{estimates2}
\norm{e^{-t\cal{A}}B u_0}_{Y(\Gamma(T{\bf M}))} \leqslant \alpha (t^{-\theta}+ 1)e^{-\beta t} \norm{u_0}_{X(\Gamma(T{\bf M}))},
\end{equation}
where the constant $0\leq \theta<1$ ($\theta =0$ and $\alpha=1$ if $X=Y$) and $0<\alpha,$ $0<\sigma < \beta$ depending only on the operator $B$ and the spaces $X, \, Y$ (for more details see the applications in Section \ref{4}).

In the rest of this paper we study the vectorial semilinear parabolic equation on the real hyperbolic manifold ${\bf M}$:
\begin{align}\label{CauchyHeat}
\begin{cases}
\partial_t u = -\cal{A}u + BG(u,t),\cr
u|_{t=0}=u_0 \in Y(\Gamma(T{\bf M})).
\end{cases}
\end{align}
To establish the existence and the decay of the solution of \eqref{CauchyHeat} we need to pose an assumption on the nonlinear part as follows. 
\begin{assumption}\label{AssG}
\item[i)] $G(u,t):Y(\Gamma(T{\bf M})) \times \r_+ \to X(\Gamma(T{\bf M}))$ is continuous and 
$$\sup_{t\geq 0}\norm{G(0,t)}_{X(\Gamma(T{\bf M}))}<\infty.$$
\item[ii)] $G(u,t)$ is locally Lipschitz, i.e. for $u_1$ and $u_2$ in the small ball of $Y$,
$$B_\rho = \left\{ u\in Y(\Gamma(T{\bf M}))| \norm{u}_{Y(\Gamma(T{\bf M}))}\leq \rho \right\},$$
we have 
$$\norm{G(u_1,t) - G(u_2,t)}_{X(\Gamma(T{\bf M}))} \leq L\norm{u_1-u_2}_{Y(\Gamma(T{\bf M}))}, \hbox{  } L>0.$$
\end{assumption}

\section{Asymptotically almost periodic mild solution: the existence and asymptotic behaviour}
\subsection{Asymptotically almost periodic function}
We recall the definitions of almost periodic and asymptotically almost periodic functions (for details see \cite{Che}).
\begin{definition}
A function  $h \in C_b(\r, X )$ is called almost periodic function if for each $ \epsilon  > 0$, there exists $l_{\epsilon}>0 $ such that every interval of length $l_{\epsilon}$ contains at least a number $T $ with the following property
\begin{equation}
 \sup_{t \in \r } \| h(t+T)  - h(t) \|_X < \epsilon.
\end{equation}
The collection of all almost periodic functions $h:\r \to X $ will be denoted by $AP(\r,X)$ which is a Banach space endowed with the norm $\|h\|_{ AP(\r,X)}=\sup_{t\in\r}\|h(t)\|_X.$
\end{definition}
To introduce the asymptotically almost periodic functions, we need the space  $C_0 (\r_+,X)$, that is, the collection of all asymptotic and continuous functions $\varphi: \r_+ \to X$ such that
$$\lim_{t \to +\infty } \| \varphi(t) \|_X = 0.$$
Clearly, $C_0 (\r_+,X)$  is a Banach space endowed with the norm $\|\varphi\|_{C_0 (\r_+,X)}=\sup_{t\in\r_+}\|\varphi(t)\|_X$
\begin{definition} 
A continuous function  $f \in C(\r_+, X )$  is said to be asymptotically almost periodic if there exist  $h \in AP(\r,X)$ and $ \varphi\in C_0(\r_+,X)$ such that
\begin{equation}
f(t) = h(t) + \varphi(t).
\end{equation}
We denote $AAP(\r_+, X):= \{f:\r_+ \to X \mid f\hbox{ is asymptotically almost periodic on $\r_+$}\}$. We have that $AAP(\r_+,X)$ is a Banach space with the norm defined by $\|f\|_{ AAP(\r_+,X)}=\|h\|_{ AP(\r, X)}+\|\varphi\|_{ C_0(\r_+,X)}$.
\end{definition}

\subsection{The linear equation}\label{3.1}
We consider the following linear equation
\begin{align}\label{CauchyHeat1}
\begin{cases}
\partial_t u + \cal{A}u = B(f)(t),\cr
u|_{t=0}=u_0 \in Y(\Gamma(T{\bf M})).
\end{cases}
\end{align}
Given $t >0$,  the mild solution of the system \eqref{CauchyHeat1} is defined by
\begin{equation}\label{mild:linear1}
u(t) := e^{-t\mathcal{A}}u_0 + \int_0^te^{-(t-\tau)\mathcal{A}} B(f)(\tau) d\tau.
\end{equation}

\begin{lemma}\label{Thm:linear}
Let $({\bf M},g)$ be a $d$-dimensional real hyperbolic manifold. Suppose 
$f \in C_b(\r_+, X)$, the system \eqref{CauchyHeat1} has one and only one mild solution $u\in  
C_b(\r_+, Y))$.
Moreover, there exists a positive constant $M$, independent of $u_0$ and $f$, such that
\begin{equation}\label{esper}
\norm{u}_{C_b(\r_+, Y)}\le \norm{ u_0}_Y + M\norm{f}_{C_b(\r_+,X)}.
\end{equation}
\end{lemma}
\begin{proof}
\def\A{\mathcal{A}}
Consider the function $u$ defined by the formula  \eqref{mild:linear1} with $u(0) = u_0$. We need to show that $\|u(t)\|_Y$ is bounded. Using the inequalities \eqref{estimates1} and \eqref{estimates2},  for each $t\in \r$,  we have  
\begin{eqnarray*}
\|u(t)\|_Y & \le& \| e^{-t\mathcal{A}}u_0\|_Y + \int_0^t \|e^{-(t-\tau)\A}B(f)(\tau)\|_Y d\tau \cr
&\le& e^{-\sigma t}\| u_0\|_Y + \alpha\int_0^t \left((t-\tau)^{-\theta} + 1\right)e^{-\beta(t-\tau)} \left\| f(\tau) \right\|_X d\tau \cr
&=& e^{-\sigma t}\| u_0\|_Y + \alpha\int_0^t \left(\tau^{-\theta} + 1\right)e^{-\beta\tau} \left\| f(t-\tau) \right\|_X d\tau \cr
&\le& e^{-\sigma t}\| u_0\|_Y + \alpha\int_0^\infty \left(\tau^{-\theta} + 1\right)e^{-\beta\tau} \left\| f(t-\tau) \right\|_X d\tau \cr
&\le& \| u_0\|_Y + \alpha\left(\beta^{\theta -1}\mathbf{\Gamma}(1-\theta)+\frac{1}{\beta}\right)\|f\|_{C_b(\r_+,X)}.
\end{eqnarray*}
The proof is completed when $M:=\alpha\left(\beta^{\theta -1}\mathbf{\Gamma}(1-\theta)+\dfrac{1}{\beta}\right)$.
\end{proof} 
\begin{remark}\label{Re34}
If we consider the linear equation 
\begin{equation}\label{whole}
\partial_t u + \mathcal{A}u = B(f)(t)
\end{equation}
on the whole time line-axis $t\in \mathbb{R}$, then we have the following definition of the mild solution
$$u(t) := \int_{-\infty}^te^{-(t-\tau)\mathcal{A}} B(f)(\tau) d\tau.$$
By the same way as in the proof of Lemma \ref{Thm:linear} we can prove that Equation \eqref{whole} has an unique mild solution such that
$$\norm{u(t)}_Y \leqslant M\|f\|_{C_b(\r,X)},$$
where $M=\alpha\left(\beta^{\theta -1}\mathbf{\Gamma}(1-\theta)+\frac{1}{\beta}\right)$.
\end{remark}
By using Lemma \ref{Thm:linear} we can define the solution operator $S: C_b(\r_+, X(\Gamma(T{\bf M})))\to C_b(\r_+,Y(\Gamma(T{\bf M})))$ of Equation \eqref{CauchyHeat1} as follows
\begin{align*}
S: C_b(\r_+,X(\Gamma(T{\bf M}))) &\rightarrow C_b(\r_+,Y(\Gamma(T{\bf M})))\cr
f&\mapsto S(f),
\end{align*}
where
\begin{equation}\label{SolOpe}
S(f)(t):= u_0+\int_0^t e^{-(t-\tau)\A}f(\tau)d\tau.
\end{equation}

Using Lemma \ref{Thm:linear} and Remark \ref{Re34} we prove a Massera-type principle to obtain the main theorem of this section. In particular,
we prove that if the external force $f$ is asymptotically almost periodic, then the mild solution of Equation \eqref{CauchyHeat} is also asymptotically almost periodic (see \cite[Theorem 3.3]{HuyXuan2020} for the case of periodic mild solution).
\begin{theorem}\label{pest}
Let $(M,g)$ be a $d$-dimensional real hyperbolic manifold and assume $p>d$. 
Let $f \in AAP(X(\Gamma(T{\bf M})))$. Then, the problem \eqref{CauchyHeat} has one and only one  asymptotically almost periodic mild solution $\hat{u}(t) \in AAP( Y(\Gamma(T{\bf M})))$ satisfying 
\begin{equation}\label{esper}
\|\hat{u}\|_{C_b(\r, Y(\Gamma(T{\bf M})))}\le \|u_0\|_p + M \|f\|_{C_b(\r_+,X(\Gamma(T{\bf M}))),}
\end{equation}
where $M:=\alpha \left(\beta^{\theta -1}\mathbf{\Gamma}(1-\theta)+\frac{1}{\beta}\right).$
\end{theorem} 
\def\xcal{\mathcal X}
\begin{proof} 
By using Assetion i) in Lemma \ref{Thm:linear}, it is fulfilled to show that the solution operator $S$ maps $AAP(\r, X(\Gamma(T{\bf M})))$ into $AAP(\r, Y(\Gamma(T{\bf M})))$. 

Indeed, for each $f\in AAP( \r_+, X(\Gamma(T{\bf M}))) $, there exist  $H \in AP(\r,X(\Gamma(T{\bf M})))$ and $ \Phi\in C_0(\r_+,X(\Gamma(T{\bf M})))$ such that  $F(t) =  H(t)+ \Phi(t)$ $(t \in \r_+)$. Using \eqref{SolOpe} we have
\begin{eqnarray*}
S(F)(t) &=& e^{-t\cal{A}}u_0 + \int_{0}^t  e^{-(t-\tau) \mathcal{A}} Bf(\tau)d\tau\cr
&=& e^{-t\cal{A}}u_0 + \int_{0}^t  e^{-(t-\tau) \mathcal{A}} B(H)(\tau)d\tau+\int_{0}^t  e^{-(t-\tau) \mathcal{A}}B(\Phi)(\tau)d\tau\cr
&=& e^{-t\cal{A}}u_0 + \int_{-\infty}^t  e^{-(t-\tau) \mathcal{A}} B(H)(\tau)d\tau + \int_{0}^t  e^{-(t-\tau) \mathcal{A}}B(\Phi)(\tau)d\tau \cr
&&- \int_{-\infty}^0 e^{-(t-\tau)\mathcal{A}}B(H)(\tau)d\tau \hbox{   for   } t>0.
\end{eqnarray*}
Putting
\begin{equation*}
\hat{S}(H)(t):=\int_{-\infty}^te^{-(t-\tau)\mathcal{A}} B(H) (\tau) d\tau
\end{equation*}
and
\begin{equation*}
\tilde{S}(\Phi)(t) := \int_{0}^t  e^{-(t-\tau) \mathcal{A}} B(\Phi) (\tau)d\tau.
\end{equation*}
As a consequence of Lemma \ref{Thm:linear} and Remark \ref{Re34} the above functions are bounded. We have
$$S(F)(t) = e^{-t\cal{A}}u(0) + \hat{S}(H)(t) + \tilde{S}(\Phi)(t) - \hat{S}(H)(0) , t\in\mathbb{R_+}.$$

We now prove that $ \hat{S}(H)\in AP(\r, Y(\Gamma(T{\bf M})))$. By considering the change of variable $\tau:= t- \tau$ we can re-write $\hat{S}(H)$ as
$$\hat{S}(H)(t) = \int_{0}^{\infty}  e^{-\tau \mathcal{A}} B(H)(t-\tau)d\tau.$$
Since $ H\in AP( \r, X(\Gamma(T{\bf M})))$, for each $ \epsilon  > 0$, there exists $l_{\epsilon}>0 $ such that every interval of length $l_{\epsilon}$ contains at least a number $T $ with the following property
$$\sup_{t \in \r } \| H(t+T)  - H(t) \|_X < \epsilon.$$
Therefore, we have
\begin{eqnarray*}
\left\|\hat{S}(H)(t+T) - \hat{S}(H)(t)\right\|_Y &=& \left\|\int_{0}^\infty e^{-\tau\mathcal{A}}B[H(t+T-\tau)- H(t-\tau)] d\tau \right\|_Y \cr
&\leqslant& \int_{0}^\infty \left\| e^{-\tau\mathcal{A}}B[H(t+T-\tau)- H(t-\tau)]\right\|_Y d\tau \cr
&\leqslant& \alpha \int_{0}^\infty (\tau^{-\theta}+1) e^{-\beta  \tau}d\tau \left\| H(\cdot+T) - H(\cdot) \right\|_{C_b(\r,X)}\cr
&\leqslant& M\left\| H(\cdot+T) - H(\cdot)\right\|_{C_b(\r,X)} \leqslant M\epsilon,
\end{eqnarray*}
for all $t \in \r_+$, where $M$ is determined as in Lemma \ref{Thm:linear}. Hence, $\hat{S}(H) \in AP( \r, Y(\Gamma(T{\bf M}))$ and the solution operator $\hat{S}$ preserves the almost periodic functions.

We remain to show that $ \tilde{S}(\Phi)(t)+e^{-t\mathcal{A}}u(0) - \hat{S}(H)(0)$ belongs to $C_0(\mathbb{R}_+, Y(\Gamma(T{\bf M})))$. To do this we will prove 
$$\mathop{\lim}\limits_{t \to \infty } \| \tilde{S}(\Phi)(t) \|_Y = 0 \hbox{  and  }  \lim_{t\to \infty}  \|e^{-t\mathcal{A}}u(0) - \hat{S}(H)(0)\|_Y = 0.$$
Indeed, we have
\begin{eqnarray*}
\tilde{S}(\Phi)(t) &=& \int_{0}^t  e^{-(t-\tau) \mathcal{A}} G(\Phi)(\tau)d\tau, t\in\mathbb{R_+} \cr
&=&\int_{0}^{t/2}  e^{-(t-\tau) \mathcal{A}}G(\Phi)(\tau)d\tau +\int_{t/2}^t  e^{-(t-\tau) \mathcal{A}}G(\Phi)(\tau)d\tau \cr
&=& S_1(\Phi)(t) +S_2(\Phi)(t).
\end{eqnarray*}
Using \eqref{estimates2} the first term $S_1(\Phi)(t)$ can be estimated as follows
\begin{eqnarray*}
\left\|S_1(\Phi)(t)\right\| _Y &\leqslant& \int_{0}^ {t/2}\left\| e^{-(t-\tau)\mathcal{A}}G(\Phi)(\tau)\right\|_Y d\tau\cr
&\leqslant& \int_{0}^ {t/2}\alpha \left((t-\tau)^{-\theta}+1\right) e^{-\beta(t-\tau)} \|\Phi\|_{C_b(\r_+,X)} d\tau\cr
&\leqslant& \int_{0}^ {t/2} \alpha \left[\left( \dfrac{2}{t}\right)^{\theta}+1\right] e^{-\beta(t-\tau)} \|\Phi\|_{C_b(\r_+,X)} d\tau \cr
&\leqslant& \alpha\left[\left( \dfrac{2}{t}\right)^{\theta}+1\right] \frac{1}{\beta} \left(e^{- \frac{\beta t}{2}}-e^{-\beta t}\right) \|\Phi\|_{C_b(\r_+,X)}. 
\end{eqnarray*}
This implies that $\mathop{\lim}\limits_{t \to \infty } \| S_1\Phi(t) \|_Y=0$. The second term $S_2(\Phi)(t)$ can be estimated as follows
\begin{eqnarray*}
\left\|S_2(\Phi)(t)\right\|_Y &\leqslant& \int_{t/2}^ {t}\left\| e^{(t-\tau)\mathcal{A}}B(\Phi)(\tau)\right\|_Y d\tau\cr
&\leqslant& \int_{t/2}^ {t}\alpha \left((t-\tau)^{-\theta}+1\right) e^{-\beta(t-\tau)} \|\Phi(\tau)\|_{C_b(\r_+,X)} d\tau. 
\end{eqnarray*}

Since $\lim_{t\to \infty}\norm{\Phi(t)}_X=0$, for all $\epsilon >0$ there exists $ t_0 $ large enough such that for all $t>t_0$, we have $\|\Phi(t)\|_X < \epsilon$. Therefore, 
$$\left\|S_2(\Phi)(t)\right\|_Y \leqslant \epsilon \int_{t/2}^ {t}\alpha \left((t-\tau)^{-\theta}+1\right) e^{-\beta(t-\tau)}  d\tau \leqslant \epsilon  M\,\,\,\, (t>2t_0),$$
where  $M$ given as in Lemma \ref{Thm:linear}. This implies that 
$$\mathop{\lim}\limits_{t \to \infty } \| S_2(\Phi)(t) \|_Y = 0.$$
Therefore
$$\mathop{\lim}\limits_{t \to \infty } \| \tilde{S}(\Phi)(t) \|_Y = 0.$$

Now we prove that $\lim_{t\rightarrow \infty} \left\| e^{-t\mathcal{A}}u(0) - \hat{S}(H)(0) \right\|_Y = 0$. Thanks to \eqref{estimates1} we have that
$$ \left\| e^{-t\mathcal{A}}u(0) \right\|_Y \leqslant e^{- \sigma t} \left\|u(0) \right\|_Y.$$
Then $\mathop{\lim}\limits_{t \to \infty } \left\| e^{-t\mathcal{A}}u(0) \right\|_Y = 0$.
On the other hand
\begin{eqnarray*}
\left\|\hat{S}(H)(0) \right\|_Y &\leqslant& \int_{-\infty}^0 \left\| e^{-(t-\tau)\mathcal{A}} B(H)(\tau) \right\|_Y d\tau \cr
&\leqslant& \int_{t}^\infty \left\| e^{-\tau\mathcal{A}} B(H) (t-\tau) \right\|_Y  d\tau\cr
&\leqslant& \int_{t}^\infty \alpha \left(\tau^{-\theta} +1 \right)e^{-\beta\tau}d\tau\left\|H\right\|_{C_b(\r_+,X)} \cr
&\leqslant& \int_{t}^\infty \alpha \left( t^{-\theta} +1 \right)e^{-\beta\tau}d\tau\left\|H\right\|_{C_b(\r_+,X)} \cr
&\leqslant& M\left\|H\right\|_{C_b(\r_+,X)},
\end{eqnarray*}
where $M$ is given as in Lemma \ref{Thm:linear}. This inequality leads to $\mathop{\lim}\limits_{t \to \infty } \left\|\hat{S}(H)(0) \right\|_Y =0$. Therefore $ \tilde{S}(\Phi)(t)+e^{-t\mathcal{A}}u_0 - \hat{S}(H)(0)$ belongs to $C_0(\r_+, Y(\Gamma(T{\bf M})))$. Our proof is now completed.
\end{proof}

\subsection{The semi-linear equation}\label{3.2}
We now prove the existence and the stability of the asymptotically almost periodic mild solution in the semilinear case.
Recall that our vectorial equations on the real hyperbolic manifold $({\bf M},g)$ are described by 
\begin{align}\label{CauchyNSE}
\begin{cases}
\partial_t u + \mathcal{A}u = BG(u,t),\\
u|_{t=0}=u_0 \in Y,
\end{cases}
\end{align}
where $G(u,t)$ satisfy Assumption \ref{AssG}.

We define the mild solution to the equations \eqref{CauchyNSE} as follows
\begin{equation}\label{MildS}
u(t) = e^{-t\mathcal{A}}u_0 + \int_0^t e^{-(t-\tau)\mathcal{A}} BG(u(\tau),\tau) d\tau.
\end{equation}
The main theorem of this section is
\begin{theorem}\label{thm2.20}
Let $({\bf M},g)$ be a $d$-dimensional real hyperbolic manifold. Suppose that $G$ satisfy Assumption \ref{AssG}. Then, the following assertions hold:
\begin{itemize}
\item[(i)] If the norm $\|u_0\|_Y$, $\sup_{t>0}\norm{G(t,0)}_X$ and $L$ are sufficiently small, 
the equation \eqref{CauchyNSE} has one and only one asymptotically almost periodic mild solution $\hat{u}$ in $C_b(\r_+, Y)$ satisfying that $\norm{\hat{u}}_{C_b(\r,Y)}\leqslant \rho$.
\item[(ii)] The asymptotically almost periodic mild solution  $\hat{u}$ of the equation \eqref{CauchyNSE} 
is exponentially stable in the sense that for any other mild solution $u\in C_b(\r_+, Y)$ to \eqref{CauchyNSE} such that $\|u(0)-\hat{u}(0)\|_{Y}$ is small enough, we have 
\begin{equation}\label{stasol}
\|u(t)-\hat{u}(t)\|_{Y}\le {C_\gamma}{e^{-\gamma t}}\|u(0)-\hat{u}(0)\|_{Y} \hbox{ for all }t>0,
\end{equation}
here $\gamma$ is a positive constant satisfying 
$0<\gamma<\min\{\frac{\sigma}{2}, \sigma - \left(\frac{\alpha L \sigma\mathbf{\Gamma}(1-\theta)}{\sigma-2\alpha L}\right)^{\frac{1}{1-\theta}}\}$, and $C_\gamma$ is a constant independent of $u$ and $\hat{u}$. 
\end{itemize}
\end{theorem} 

\begin{proof}
(i): Putting
\begin{equation}\label{bro}
\B^{AAP}_\rho:=\{v\in C_b(\r_+, Y): v \hbox{  is asymptotically almost periodic and  } \|v\|_{C_b(\r_+,Y)} \le \rho \}.
\end{equation}

We consider the equation 
\begin{equation}\label{ns1}
    \partial_tu  + \A u  = BG(v,t).
 \end{equation} 
Applying Lemma \ref{Thm:linear} for the right-hand side  $G(u,t)$ instead of $f(t)$, we obtain that given $v\in \B^{AAP}_\rho$, there exists a unique asymptotically almost periodic mild solution $u\in AAP(\r_+, Y)$  to \eqref{ns1} satisfying
\begin{eqnarray}\label{ephi}
\|u\|_{\infty,Y}&\le& \| u_0\|_Y + M\sup_{t>0}\|G(v(t),t)\|_{X} \cr
&\le& \| u_0\|_Y + M \left( \sup_{t>0}\|G(v(t),t)-G(0,t)\|_{X} + \sup_{t>0}\norm{G(0,t)}_X \right)\cr
&\le& \| u_0\|_Y + M \left(L \sup_{t>0}\norm{v(t)}_Y + \sup_{t>0}\norm{G(0,t)}_X \right)\cr
&\le& \| u_0\|_Y + M \left(L\rho + \sup_{t>0}\norm{G(0,t)}_X \right).
\end{eqnarray}
By Formula \eqref{mild:linear1} with $G(v,t)$ instead of $f(t)$, we have
\begin{equation}\label{defphi1}
u(t) = e^{-t\mathcal{A}}u_0 + \int_0^t e^{-(t-\tau)\mathcal{A}} G(v(\tau),\tau)  d\tau.
\end{equation}
Setting
\begin{equation}\label{defphi}
\begin{split}
\Phi(v)&:=u
\end{split}\,.
\end{equation}
If $\| u_0\|_Y$, $\rho$, $L$ and $\sup_{t>0}\norm{G(0,t)}_X$ are small enough, the inequality \eqref{ephi} leads to $\Phi(\B^{AAP}_\rho) \subset \B^{AAP}_\rho$. On the other hand, for $v_1, v_2\in \B^{AAP}_\rho$, the function  
$z:=\Phi(v_1)-\Phi(v_2)$ becomes the unique strong mild solution to the equation 
$$\partial_tz  + \A z   = B[G(v_1,t) - G(v_2,t)].$$
By using again Lemma \ref{Thm:linear} we have
\begin{eqnarray}\label{ephi3}
\|\Phi(v_1)-\Phi(v_2)\|_{C_b(\r_+,Y)}&\le& ML\|(v_1-v_2)\|_{C_b(\r_+,X)}.
\end{eqnarray}
Therefore, $\Phi $ is a contraction from $\B^{AAP}_\rho$ into $\B^{AAP}_\rho$ if $L$ is sufficiently small. For these values of  $\rho$, $L$, $\|u_0 \|_Y$ and $\sup_{t>0}\norm{G(0,t)}_X$ the fixed point argument shows that there exists a unique point $\hat{u} \in \B^{AAP}_\rho$ such that $\Phi(\hat{u})=\hat{u}$.  On account of the above definition of $\Phi$, this function $\hat{u}$ is the asymptotically almost periodic mild solution to \eqref{CauchyNSE} in the ball $\B^{AAP}_\rho$ on the real hyperbolic manifold ${\bf M}$ . 
 
\medskip
(ii): We will prove the exponential stability of the asymptotically almost periodic mild solution $\hat{u}$ to Equation \eqref{CauchyNSE} by using  the cone inequality theorem which we now recall. To this purpose,  we first recall the  notion of 
a cone in a Banach space:
A closed subset $\cal K$ of a Banach space $W$ is called a {\it cone} if 
\begin{enumerate}
\item[$\bullet$] $x_0\in \cal K\Longrightarrow\lambda x_0\in \cal K$ for all $\lambda\ge 0$;
\item[$\bullet$] $x_1,\ x_2\in \cal K \Longrightarrow x_1+ x_2\in \cal K$;
\item[$\bullet$] $\pm x_0\in \cal K\Longrightarrow x_0=0$. 
\end{enumerate}
Then, let a cone $ \cal K$ be given in the Banach space $W$. For $x, y\in W$ we then write
$x\le y$ if $y-x\in \cal K$. 

If the cone $ \cal K$ is invariant under a linear operator $A$, 
then it is easy to see that $A$ preserves the inequality, i.e., $x\le y$ implies $Ax\le Ay$. 
Now, the following cone inequality is proven in \cite[Theorem I.9.3]{DalKre}.

{\it Suppose that
 $\cal K$ is a cone in a Banach space $W$ such that $\cal K$ is invariant under a bounded linear operator  $A\in \mathcal{L}(W)$ having spectral radius $r_A<1$. 
 For a vector $x\in W$ satisfying 
\begin{equation}\label{ci} 
x\le Ax+z\hbox{ for some given }z\in W,
\end{equation}
 we have that it also satisfies the estimate $x\leqslant y$, where $y\in W$ is a solution of the equation $y=Ay+z$.}

To prove the stability of the mild solution $\hat u$ we
let $u(t)$ be  any bounded solution of Equation 
\eqref{MildS} corresponding to initial  value  
$u_0:=u(0)\in B_{\frac{\rho}{2}}:=\{v\in Y : \|v\|_{Y}\le \frac \rho2\}$. 
Then, we have that
\begin{equation}\label{DE}
u(t)-\hat u(t)=e^{-t\A}(u(0)-\hat u(0))+
\int_{0}^t e^{-(t-\tau)\A}B[G(\tau, u)- G(\tau,\hat{u})]d\tau
\end{equation}
for $t\geqslant 0$.

This follows that
\begin{eqnarray*}
\|u(t)-\hat u(t)\|_{Y}&\le& e^{-\sigma t}\|u(0)-\hat u(0)\|_{Y} \cr
&&+ \int_{0}^t \alpha\left((t-\tau)^{-\theta} +1\right)e^{-\beta(t-\tau)}\|G(\tau,u(\tau))-G(\tau,\hat u(\tau))\|_X d\tau\cr
&\le& e^{-\sigma t}\|u(0)-\hat u(0)\|_{Y} \cr
&&+ L \int_{0}^t \alpha\left( (t-\tau)^{-\theta}+1\right)e^{-\beta(t-\tau)}\|u(\tau)-\hat{u}(\tau)\|_Yd\tau \hbox{ for  $t\ge 0$}
\end{eqnarray*}
(here we use  the fact that $0<\sigma<\beta$).

Put $\phi(t)=\|u(t)-\hat u(t)\|_{Y}$. 
Then $\sup_{t\ge 0}\phi(t)<\infty$, and 
\begin{equation}\label{phi1}
\phi(t)\le e^{-\sigma t}\|u(0)-\hat u(0)\|_Y+ 
L \int_{0}^t \alpha\left((t-\tau)^{-\theta}+1\right)e^{-\sigma(t-\tau)}\phi(\tau)d\tau\hbox{ for  }t\ge 0.
\end{equation}
We will use the cone inequality theorem applying to Banach space $W:=L^\infty([0,\infty))$  
which is the space of real-valued functions defined and essentially bounded on $[0,\infty)$ 
(endowed with the sup-norm denoted by $\|\cdot\|_\infty$) with the cone $\cal K$ being the set of all (a.e.) nonnegative functions.  We now consider the linear operator 
$A$ defined for $h\in W$ by
$$(Ah)(t)=L \int_{0}^t \alpha\left((t-\tau)^{-\theta}+1\right)e^{-\sigma(t-\tau)}h(\tau)d\tau\hbox{ for  }t\ge 0.$$ 
Then, we have that
$$\sup_{t\ge 0}|(Ah)(t)|=\sup_{t\ge 0}L \int_{0}^t \alpha\left((t-\tau)^{-\theta}
+ 1\right) e^{-\sigma(t-\tau)}|h(\tau)|d\tau$$
$$\le L\alpha\left(\sigma^{\theta-1}\mathbf{\Gamma}(1-\theta)
+\frac{1}{\sigma}\right)\|h\|_{\infty}$$
where $\mathbf{\Gamma}$ is the gamma function.

Therefore, $A\in {\cal L}(L^\infty([0,\infty)))$ and 
$\|A\|\le L \alpha\left(\sigma^{\theta-1}\mathbf{\Gamma}(1-\theta)
+\frac{1}{\sigma}\right)<1$ for $L$ being small enough. Note that if 
$\|u(0)-\hat{u}(0)\|_Y$ is small enough, by the same way as in the proof of Assertion i), we can show that the solution $u(t)-\hat{u}(t)$ of \eqref{DE} exists and unique in the ball $\B_\rho$ of $C_b(\r_+,Y)$. Therefore, we have that $\phi$ belongs to $\cal K$.
  
Obviously, $A$ leaves the cone $\cal K$ invariant.
The inequality \eqref{phi1} can now be rewritten as
$$\phi\le A\phi+z\hbox{ for }z(t)=e^{-\sigma t}\|u(0)-\hat{u}(0)\|_Y;\; t\ge 0.$$
Hence, by cone inequality theorem \ref{ci} we obtain that $\phi\le \psi$, where
$\psi$ is a solution in $L^\infty([0,\infty))$ of the equation $\psi= A\psi+z$ which can be rewritten as
\begin{equation}\label{psi}
\psi(t)= e^{-\sigma t}\|u(0)-\hat{u}(0)\|_Y+ 
L \int_{0}^t \alpha\left( (t-\tau)^{-\theta}
+{1}\right) e^{-\sigma(t-\tau)}\psi(\tau)d\tau\hbox{ for  }t\ge 0.
\end{equation}
In order to estimate $\psi$,
for $0<\ga<\min\{\frac{\sigma}{2}, \sigma - \left(\frac{\alpha L \sigma\mathbf{\Gamma}(1-\theta)}{\sigma-\alpha L}\right)^{\frac{1}{1-\theta}}\}$ 
 we set $w(t):=e^{\ga t}\psi(t)$,\ $t\ge 0$. Then, by Equality \eqref{psi} 
we obtain that
\begin{equation}\label{cone}
w(t)=e^{-(\sigma-\ga)t}\|u(0)-\hat{u}(0)\|_Y+ 
L \int_{0}^t \alpha\left((t-\tau)^{-\theta}
+1\right) e^{-(\sigma-\ga)(t-\tau)}w(\tau)d\tau\hbox{ for  }t\ge 0.
\end{equation}
We next consider the linear operator 
$D$ defined for $\varphi\in L^\infty([0,\infty))$ by
$$(D\varphi)(t)= L\int_{0}^t \alpha\left( (t-\tau)^{-\theta}+1\right)e^{-(\sigma-\ga)(t-\tau)}\varphi(\tau)d\tau\hbox{ for  }t\ge 0.$$
Again, we can estimate 
\begin{eqnarray*}
\sup_{t\ge 0}|(D\varphi)(t)|&= & \sup_{t\ge 0}L\int_{0}^t \alpha\left((t-\tau)^{-\theta}
+1\right)e^{-(\sigma-\ga)(t-\tau)}|\varphi(\tau)|d\tau\cr
&\le &  L\alpha\left((\sigma-\ga)^{\theta-1}\mathbf{\Gamma}(1-\theta)
+\frac{1}{\sigma-\gamma}\right)\|\varphi\|_\infty.
\end{eqnarray*}
Therefore, $D\in {\cal L}(L^\infty([0,\infty)))$ and 
$\|D\|\le L\alpha\left((\sigma-\ga)^{\theta-1}\mathbf{\Gamma}(1-\theta)
+\frac{1}{\sigma-\gamma}\right).$ 
The Equation \eqref{cone} can now be rewritten as
$$w= Dw+z\hbox{ for }z(t)= e^{-(\sigma-\ga)t}\|u(0)-\hat{u}(0)\|_Y,\; t\ge 0.$$
Since  $0<\ga<\min\{\frac{\sigma}{2}, \sigma - \left(\frac{\alpha L \sigma\mathbf{\Gamma}(1-\theta)}{\sigma-2\alpha L}\right)^{\frac{1}{1-\theta}}\}$
we obtain 
$$ \|D\|\le   \alpha L \left((\sigma-\ga)^{\theta-1}\mathbf{\Gamma}(1-\theta)
+\frac{2}{\sigma}\right) <1.$$
Therefore, the equation $w=Dw+z$ is uniquely solvable in $L^\infty([0,\infty))$, 
and its solution is $w=(I-D)^{-1}z$. Hence,  we  obtain that
\begin{eqnarray*}
\|w\|_\infty &=& \|(I-D)^{-1}z\|_\infty\le\|(I-D)^{-1}\|\|z\|_\infty\le 
\frac{1}{1-\|D\|}\|u(0)-\hat{u}(0)\|_p\cr
&\le& \dfrac{1}{1- \alpha L \left((\sigma-\ga)^{\theta-1}\mathbf{\Gamma}(1-\theta)
+\frac{2}{\sigma}\right)}\|u(0)-\hat{u}(0)\|_Y:=C_\ga\|u(0)-\hat{u}(0)\|_Y.
\end{eqnarray*}
This yields that
$$w(t)\le C_\ga\|u(0)-\hat{u}(0)\|_Y \hbox{ for }t\ge 0.$$
Hence, $\psi(t)=e^{-\ga t}w(t)\le C_\ga e^{-\ga t}\|u(0)-\hat{u}(0)\|_Y$. Since
$\|u(t)- \hat u(t))\|_Y=\phi(t)\le \psi(t)$, we obtain that
$$\|u(t)- \hat u(t)\|_Y\le C_\ga e^{-\ga t}\|u(0)-\hat{u}(0)\|_Y \hbox{ for }t\ge 0.$$
\end{proof}

\section{Applications}\label{4}
\subsection{Navier-stokes equation}\label{4.1}
We follow \cite{Pi,Tay} to express the Navier-Stokes equation on the real hyperbolic manifold ${\bf M}$. The imcompressible Navier-Stokes equations on ${\bf M}$ are described by the following equation
\begin{align}\label{Cauchy}
\begin{cases}
\partial_t u + \nabla_u u + \nabla p = Lu + \mathrm{div}F \\
\mathrm{div} u = 0,\\
u|_{t=0} = u_0 \in \Gamma(T{\bf M}),
\end{cases}
\end{align}
where $u=u(x,t)$ is considered as a vector field on ${\bf M}$, i.e., 
$u(\cdot,t)\in \Gamma(T{\bf M})$ and $p=p(x,t)$ is the pressure field, $L$ is the stress tensor, $\dive F$ is the external force.

Since $\mathrm{div}u=0$ it follows that $\nabla_uu = \mathrm{div} (u\otimes u)$. Therefore we obtain the equivalent system of \eqref{Cauchy} as follows
\begin{align}\label{Cauchy'}
\begin{cases}
\partial_t u + \mathrm{div}(u\otimes u) + \nabla p = Lu + \mathrm{div}F \\
\mathrm{div} u = 0,\\
u|_{t=0} = u_0 \in \Gamma(T{\bf M}).
\end{cases}
\end{align}

By the Kodaira-Hodge decomposition, an $L^2$-form can be decomposed on $M $ as 
$$L^2(\Gamma(T^*({\bf M}))) = \overline{\mathrm{Image}\, d} \oplus \overline{\mathrm{Image}\, d^*} \oplus \mathcal{H}^1({\bf M}) \, ,$$
where $\mathcal{H}^1({\bf M})$ is the space of $L^2$ harmonic $1-$forms. Taking the divergence of Equation \eqref{Cauchy'} and noting that $\mathrm{div}(\overrightarrow{\Delta}u) = \mathrm{div}(-(d-1)u) = 0$ if $\mathrm{div}\, u=0$, we get
$$\Delta_g p + \mathrm{div}[\nabla_uu] = 0,$$
where $\Delta_g$ is the Laplace-Beltrami operator on ${\bf M}$,i.e
$$\Delta_g  = \dive\mathrm{grad} = \frac{1}{\sqrt{|g|}}\frac{\partial}{\partial x^j}\left( \sqrt{|g|}g^{ij}\frac{\partial}{\partial x^i}  \right).$$

\def\dive{\operatorname{div}}
We need to choose a solution in $L^p$ of this elliptic equation. Since the spectral of $\Delta_g$ on the hyperbolic manifold ${\bf M} = \mathbb{H}^d(\mathbb{R})$ is $\left[ \frac{(d-1)^2}{4},\infty\right)$ which does not contain $0$, we have that $\Delta_g \,: \, W^{2,r} \rightarrow L^r$ is an isomorphism for $2\leq r<\infty$. This leads to
$$\mathrm{grad}p = \mathrm{grad}(-\Delta_g)^{-1}\mathrm{div}[\nabla_uu].$$
By putting  
$\mathbb{P}:= I + \mathrm{grad}(-\Delta_g)^{-1} \mathrm{div}$ ($\mathbb{P}$ is called the Kodaira-Hodge operator), we can get rid of the 
pressure term $p$ and then obtain from \eqref{Cauchy'} that 
\begin{align}\label{DivNavierStokes}
\begin{cases}
\partial_t u = Lu -\mathbb{P}\dive (u\otimes u) + \mathbb{P}\dive F,\\
\dive u = 0,\\
u|_{t=0} = u_0 \in \Gamma(T{\bf M}),
\end{cases}
\end{align}
where $L= -(d-1)+ \overrightarrow{\Delta}$.

Since the $L^p$-boundedness of Riesz transforms on a real hyperbolic manifold  (see \cite{Loho}), the operator $\mathbb{P}$ is bounded. Moreover, we would like to remark that $\mathbb{P}$ is commuted with $\overrightarrow{\Delta}$, then $e^{-t\cal{A}}$ in the real hyperbolic spaces. Therefore, we can eliminate $\mathbb{P}$ in the sense of the application of abstract results in Section 3.

Putting $B=\mathbb{P}\dive$, $G(u,t) = (u\otimes u + F)(t)$, $X(\Gamma(T{\bf M}))=L^{p/2}(\Gamma(T{\bf M}\otimes T{\bf M}))$, $Y(\Gamma(T{\bf M}))=L^p(\Gamma(T{\bf M}))$ and
$$0<\theta = \frac{d}{2}\left( \frac{1}{p} + \frac{1}{d} \right)<1 \hbox{  with  } p>d,$$
$$\sigma = d-1 + \gamma_{p,p}, \, \beta = d-1 + \frac{\gamma_{p,p}+ \gamma_{p/2,p}}{2}.$$
We have that $G$ satisfy Assumption \ref{AssG} with $L=2\rho$. Indeed, using Holder's inequality we have for $u_1,\, u_2\in B_\rho$:
\begin{eqnarray*}
\norm{G(u_1,t)-G(u_2,t)}_X &=& \norm{u_1\otimes u_1 - u_2\otimes u_2}_X \cr
&\leq& \norm{u_1\otimes (u_1-u_2)}_X + \norm{(u_1-u_2)\otimes u_2}\cr
&\leq& (\norm{u_1}_Y + \norm{u_2}_Y)\norm{u_1-u_2}_Y\cr
&\leq& 2\rho \norm{u_1-u_2}_Y.
\end{eqnarray*}
$$\norm{G(0,t)}_X = \norm{F(t)}_X<\infty$$
if we assume that $F\in C_b(\r_+,X)$.

Therefore, we can apply Theorem \ref{thm2.20} we obtain the following result.
\begin{theorem}
Let $({\bf M},g)$ be a $d$-dimensional real hyperbolic manifold. Suppose that $F\in C_b(\r_+, L^{p/2}(\Gamma(T{\bf M}\otimes T{\bf M}))$ with $p>d$.
Then, the following assertions hold true.
\begin{itemize}
\item[(i)] If the norm $\|u_0\|_{p}$ and $\|F\|_{\infty,\frac{p}{2}}$ are sufficiently small, 
Equation \eqref{DivNavierStokes} has one and only one asymptotically almost periodic mild solution $\hat{u}$ on a small ball of $C_b(\r_+, L^p(\Gamma(T{\bf M})))$.

\item[(ii)] For any other mild solution $u\in C_b(\r_+, L^p(\Gamma(T{\bf M})))$ to \eqref{Cauchy'} such that $\|u(0)-\hat{u}(0)\|_p$ is small enough, we have
\begin{equation*}
\|u(t)-\hat{u}(t)\|_p \le {C_\gamma}{e^{-\gamma t}}\|u(0)-\hat{u}(0)\|_p \hbox{ for all }t>0,
\end{equation*}
here $\gamma$ is a positive constant satisfying 
$0<\gamma<\min\{\frac{\sigma}{2}, \sigma - \left(\frac{2\alpha \rho \sigma\mathbf{\Gamma}(1-\theta)}{\sigma-4\alpha \rho}\right)^{\frac{1}{1-\theta}}\}$, and $C_\gamma$ is a constant independent of $u$ and $\hat{u}$.
\end{itemize}
Here, we denote that $\norm{.}_{p} = \norm{.}_{L^p(\Gamma(T{\bf M})))}$, $\norm{.}_{\infty,\frac{p}{2}} = \norm{.}_{L^{p/2}(C_b(\r_+,\Gamma(T{\bf M}\otimes T{\bf M})))}$.
\end{theorem}

\subsection{Semilinear vectorial heat equation}\label{4.2}
The semilinear vectorial heat equation with rough coefficients on the real hyperbolic manifold ${\bf M}$ has the following form:
\begin{align}\label{CauchyHeat11}
\begin{cases}
\partial_t u = Lu + |u(t)|^{k-1}u(t) + f(t),\cr
u|_{t=0}=u_0,
\end{cases}
\end{align}
where $L=-(d-1) + \overrightarrow{\Delta}$, $k\in \mathbb{N}$, $k\geq 2$ and $u(t,.)\in \Gamma(T{\bf M})$.
We denoted that $|.|$ is the norm defined on the hyperbolic space $(M,g)$ by 
$$|u(t)| = \left< u(t),u(t) \right>_g^{1/2}.$$

Putting
$$B=\mathrm{Id}, \, G(u,t) = |u(t)|^{k-1}u(t)+ f(t), \, X(\Gamma(T{\bf M}))=Y(\Gamma(T{\bf M}))= L^p(\Gamma(T{\bf M})),$$
and
$$\theta=0,\, \sigma = \beta = d-1 + \gamma_{p,p}.$$
For $u_1,\, u_2\in B_\rho$ and $k\geq 2$ we have
\begin{eqnarray*}
\norm{G(u_1,t) - G(u_2,t)}_X &=& \norm{|u_1|^{k-1}u_1 - |u_2|^{k-1}u_2}_Y \cr
&\leq& \sum_{j=0}^{k-1} \norm{u_1-u_2}_Y\norm{u_1}_Y^j\norm{u_2}_Y^{k-1-j} \cr
&\leq& k\rho^{k-1}\norm{u_1-u_2}_Y.
\end{eqnarray*}
Therefore, $G(u,t)$ satisfy Assumption \ref{AssG} with $L=k\rho^{k-1}$ and $f\in C_b(\r_+,X)$.

Applying Lemma \ref{Thm:linear} and Theorem \ref{thm2.20} with the above settings we obtain the following result.
\begin{theorem}
Let $({\bf M},g)$ be a $d$-dimensional real hyperbolic manifold. Suppose that $f\in C_b(\r_+, L^{p}(\Gamma(T{\bf M}))$ ($p>d$).
Then, the following assertions hold true.
\begin{itemize}
\item[(i)] If the norm $\|u_0\|_{p}$, and $\norm{f}_{\infty,p}$ are sufficiently small, 
Equation \eqref{CauchyHeat11} has one and only one asymptotically almost periodic mild solution $\hat{u}$ on a small ball of  
$C_b(\r_+, L^p(\Gamma(T{\bf M})))$.

\item[(ii)] For any other mild solution $u\in C_b(\r_+, L^p(\Gamma(T{\bf M})))$ to \eqref{Cauchy'} such that $\|u(0)-\hat{u}(0)\|_p$ is small enough, we have
\begin{equation*}
\|u(t)-\hat{u}(t)\|_p \le {C_\gamma}{e^{-\gamma t}}\|u(0)-\hat{u}(0)\|_p \hbox{ for all }t>0,
\end{equation*}
here $\gamma$ is a positive constant satisfying 
$0<\gamma<\min\{\frac{\sigma}{2}, \sigma - \frac{k\alpha \rho^{k-1} \sigma}{\sigma-2k\alpha \rho^{k-1}} \}$, and $C_\gamma$ is a constant independent of $u$ and $\hat{u}$.

\end{itemize}
Here, we denote that $\norm{.}_{p} = \norm{.}_{L^p(\Gamma(T{\bf M}))}$, $\norm{.}_{\infty,p} = \norm{.}_{L^p(C_b(\r_+,\Gamma(T{\bf M})))}$.
\end{theorem}
\begin{proof}
We apply the results in Lemma \ref{Thm:linear} and Theorem \ref{thm2.20} with noting that $\Gamma(1)=1$ and $L=k\rho^{k-1}$.
\end{proof}

\end{document}